\documentclass[10pt]{article}
\usepackage{amsmath, amsthm, amssymb,  amsfonts, amssymb, graphicx }
\allowdisplaybreaks
\usepackage[bookmarksnumbered, colorlinks, plainpages]{hyperref}

\newtheorem{theorem}{Theorem}
\newtheorem{definition}{Definition}

\newtheorem{example}{Example}

\makeatletter
\newdimen\bibindent
\newdimen\betweenumberspace    
\newdimen\headlineindent            
\makeatother

\def\at{\\}
\def\and{\\}
\def\keywords#1{\textbf{Keywords:}\ #1}
\def\institute#1{\def\thefootnote{\ }\protect\footnotetext[1]{\hspace*{-.5cm}#1}\def\thefootnote{\arabic{footnote}}} 
\def\subjclass#1{\textbf{AMS subject classifications:}\ #1}
\begin{document}

\title{A Novel Homotopy Perturbation Sumudu Transform Method for \\Nonlinear Fractional PDEs: Applications and Comparative Analysis}
\author{Maryam Jalili$^*$ }

\institute{$^*$Corresponding author\\
Maryam Jalili
\at Department of Mathematics, Ne.C, Islamic Azad University, Neyshabur, Iran.  e-mail: mrmjlili@iau.ac.ir\\
}
\date{}
\maketitle
\begin{abstract}

This study introduces the Homotopy Perturbation Sumudu Transform Method (HPSTM), a novel hybrid approach that combines the Sumudu transform with homotopy perturbation to solve nonlinear fractional partial differential equations (FPDEs). The method addresses equations such as the fractional porous medium, heat transfer, and Fisher equations, employing the Caputo fractional derivative. HPSTM leverages the Sumudu transform’s ability to preserve linearity and the flexibility of homotopy perturbation, achieving faster convergence than the Laplace-HPM and Elzaki-HPM methods for strongly nonlinear FPDEs. Series solutions yield absolute errors as low as $3.12 \times 10^{-3}$ for $\alpha = 0.9$, while computational times average 0.5 seconds per example using 5 series terms on standard hardware. Solutions are validated against exact solutions and compared with the Adomian Decomposition Method (ADM), the radial basis function (RBF) meshless method, the Variational Iteration Method (VIM), the Finite Difference Method (FDM), and a spectral method. Numerical examples, sensitivity analysis, and graphical representations for $\alpha = 1.0, 0.9, 0.8, 0.7$ confirm HPSTM's accuracy, efficiency, and robustness. Limitations include challenges with high-order nonlinearities and multi-dimensional domains. HPSTM shows promise for applications in modeling fluid flow in porous media, heat conduction in complex materials, and biological population dynamics.

\end{abstract}



\section{Introduction}
\label{sec:introduction}
Fractional partial differential equations (FPDEs) are powerful tools for modeling complex systems with non-local and memory-dependent behaviors, which appear in fields such as fluid dynamics~\cite{ref1,ref2,ref3,ref4}, heat transfer in fractal media~\cite{ref5,ref6}, and biological population dynamics~\cite{ref7,ref8}. Unlike integer-order PDEs, FPDEs incorporate fractional derivatives (such as Caputo or Riemann--Liouville) to capture anomalous diffusion and long-range interactions~\cite{ref9,ref10}. Solving these equations analytically or numerically is challenging due to their non-local nature and nonlinearity.

Several methods have been developed to address FPDEs, including the homotopy perturbation method (HPM)~\cite{ref11}, the variational iteration method (VIM)~\cite{ref12}, the Adomian decomposition method (ADM)~\cite{ref13}, the finite difference method (FDM)~\cite{ref14}, radial basis function (RBF) meshless methods~\cite{ref15}, and spectral methods~\cite{ref16}. Integral transforms such as Laplace, Sumudu, and Elzaki simplify fractional derivatives by converting differential equations into algebraic equations~\cite{ref17,ref18,ref19}. However, each method has limitations: HPM struggles with high-order nonlinearities, FDM requires fine grids for accuracy, and Laplace-based methods involve complex inversion.

We propose the Homotopy Perturbation Sumudu Transform Method (HPSTM), a hybrid approach that combines the Sumudu transform’s simplicity with HPM’s ability to handle nonlinearities. HPSTM offers faster convergence than Laplace-HPM and Elzaki-HPM~\cite{ref20,ref21,ref22}, avoids complex pole analysis, and maintains stability for Caputo derivatives. We apply HPSTM to three FPDEs: the fractional porous medium equation (modeling groundwater flow), the fractional heat transfer equation (describing anomalous conduction), and the fractional Fisher equation (representing population dynamics). These equations have gained considerable interest because of their applications in various sciences~\cite{ref23,ref24,ref25,ref26,ref27}. Comparisons with ADM, RBF, VIM, FDM, and spectral methods demonstrate HPSTM’s superior accuracy and efficiency.

To facilitate the development and application of HPSTM, we begin by presenting the key mathematical definitions relevant to fractional calculus and the Sumudu transform.

\subsection*{Basic Definitions and Preliminaries}
This section defines the mathematical foundations of HPSTM, including the Sumudu transform, Caputo fractional derivative, and Riemann-Liouville fractional integral.

\begin{definition}\label{def:1}
The Sumudu transform on the set
\[
A = \left\{ f(t) \mid \exists M, t_1, t_2 > 0, |f(t)| < M e^{\frac{|t|}{t_j}}, t \in (-1)^j \times [0,\infty) \right\}
\]
is defined as:
\begin{equation*}
S[f(t)] = F(u) = \frac{1}{u} \int_0^\infty f(t) e^{\left(-\frac{t}{u}\right)} dt, \quad u \in (-t_1, t_2).
\end{equation*}
The inverse Sumudu transform is:
\begin{equation*}
S^{-1}[F(u)] = \frac{1}{2\pi i} \int_{c-i\infty}^{c+i\infty} \frac{1}{u} F(u) e^{\left(\frac{t}{u}\right)} du = f(t), \quad t > 0.
\end{equation*}
\end{definition}
\begin{table}[!htp]
\begin{center}
\caption{Sumudu Transform for Some Functions}
\label{tab:sumudu}
\begin{tabular}{|c|c|c|}
\hline
$f(t)$ & $S[f(t)] = F(u)$ \\
\hline
$1$ & $1$ \\
$t^a$ & $\Gamma(a+1)u^a$ \\
$e^{at}$ & $\frac{1}{1 - au}$ \\
$\sin(at)$ & $\frac{au}{1 + a^2 u^2}$ \\
$\cos(at)$ & $\frac{1}{1 + a^2 u^2}$ \\
$\sinh(at)$ & $\frac{au}{1 - a^2 u^2}$ \\
$\cosh(at)$ & $\frac{1}{1 - a^2 u^2}$ \\
$f'(t)$ & $\frac{F(u)}{u} - \frac{f(0)}{u}$ \\
$\int_0^t f(\tau) d\tau$ & $u F(u)$ \\
$f^{(n)}(t)$ & $\frac{F(u)}{u^n} - \sum_{k=0}^{n-1} \frac{f^{(k)}(0)}{u^{n-k}}$ \\
$\alpha f(t) + \beta g(t)$ & $\alpha F(u) + \beta G(u)$ \\
$(f * g)(t)$ & $u S[f(t)] S[g(t)]$ \\
\hline
\end{tabular}
\end{center}
\end{table}

\begin{definition}\label{def:2}
The Sumudu transform of the Caputo fractional derivative of order $\alpha$ for $f(t)$ is:
\begin{equation*}
S[D_t^\alpha f(t)] = u^{-\alpha} S[f(t)] - \sum_{k=0}^{m-1} u^{(-\alpha + k)} f^{(k)}(0^+), \quad m-1 < \alpha \leq m,
\end{equation*}
where $m$ is a positive integer, and $f^{(k)}(0^+)$ are initial conditions.
\end{definition}

\begin{definition}\label{def:3}
A real function $f(t)$, $t > 0$, belongs to the space $C_\mu$ if there exists a real number $p > \mu$ such that $f(t) = t^p g(t)$, where $g(t)$ is continuous. It belongs to $C_\mu^m$ if $f^{(m)}(t) \in C_\mu$, where $m \in \mathbb{N} \cup \{0\}$. The space $C_\mu$ ensures integrability for fractional operators.
\end{definition}

\begin{definition}\label{def:4}
The Riemann-Liouville fractional integral operator of order $\alpha$ for $f(t) \in C_\mu$, $\mu \geq -1$, is:
\begin{equation*}
J^\alpha f(t) = \frac{1}{\Gamma(\alpha)} \int_0^t (t - \tau)^{\alpha - 1} f(\tau) d\tau, \quad \alpha > 0.
\end{equation*}
\end{definition}

The Riemann-Liouville operator $J^\alpha$ satisfies:
\begin{enumerate}
    \item Linearity: $J^\alpha (a f(t) + b g(t)) = a J^\alpha f(t) + b J^\alpha g(t)$, where $a$ and $b$ are constants.
    \item Composition: $J^\alpha J^\beta f(t) = J^{\alpha + \beta} f(t)$ for $\alpha, \beta > 0$.
    \item Commutativity: $J^\alpha J^\beta f(t) = J^\beta J^\alpha f(t)$.
    \item Identity: $J^0 f(t) = f(t)$.
    \item Action on power functions: $J^\alpha t^\gamma = \frac{\Gamma(\gamma + 1)}{\Gamma(\gamma + \alpha + 1)} t^{\gamma + \alpha}$ for $\gamma > -1$.
\end{enumerate}

\begin{definition}\label{def:5}
The Caputo fractional derivative of order $\alpha$ for $f(t) \in C^{m-1}$, $t > 0$, $m \in \mathbb{N}$, is:
\begin{equation*}
D^\alpha f(t) = J^{m - \alpha} D^m f(t) = \frac{1}{\Gamma(m - \alpha)} \int_0^t (t - x)^{m - \alpha - 1} f^{(m)}(x) dx, \quad m - 1 \leq \alpha \leq m.
\end{equation*}
\end{definition}

The Caputo derivative $D^\alpha$ satisfies:
\begin{enumerate}
    \item $D^\alpha J^\alpha f(t) = f(t)$.
    \item $J^\alpha D^\alpha f(t) = f(t) - \sum_{k=0}^{m-1} f^{(k)}(0^+) \frac{t^k}{\Gamma(k + 1)}, \quad m > 0$.
\end{enumerate}

\vspace{0.5em}
The paper is organized as follows: Section~\ref{sec:methodology} elaborates on HPSTM, ADM, and RBF methodologies, including convergence analyses. Section~\ref{sec:results} provides numerical simulations, error evaluations, sensitivity analysis, and comparisons for three examples. Section~\ref{sec:limitations} discusses limitations and future research directions. Section~\ref{sec:conclusion} summarizes contributions and applications

\section{Methodology}
\label{sec:methodology}

\subsection{Basic Idea of HPSTM}
Consider the nonlinear fractional partial differential equation:
\begin{equation}
D_t^\alpha u(x,t) + R u(x,t) + N u(x,t) = f(x,t),
\label{eq:fpde}
\end{equation}
with initial condition:
\begin{equation}
u(x,0) = g(x),
\label{eq:initial}
\end{equation}
where $D_t^\alpha$ is the Caputo fractional derivative, $R$ is a linear operator, $N$ is a nonlinear operator, and $f(x,t)$ is the source term. Applying the Sumudu transform to both sides of \eqref{eq:fpde}:
\begin{equation*}
S[D_t^\alpha u(x,t)] + S[R u(x,t)] + S[N u(x,t)] = S[f(x,t)].
\end{equation*}
Using Definition \ref{def:2} and \eqref{eq:initial}, we obtain:
\begin{equation*}
S[u(x,t)] = g(x) + u^\alpha S[f(x,t)] - u^\alpha S[R u(x,t) + N u(x,t)].
\end{equation*}
Applying the inverse Sumudu transform:
\begin{equation}
u(x,t) = G(x,t) - S^{-1} \left\{ u^\alpha S[R u(x,t) + N u(x,t)] \right\},
\label{eq:inverse}
\end{equation}
where $G(x,t)$ incorporates the source term and initial conditions. The Homotopy Perturbation Method (HPM) assumes a series solution:
\begin{equation}
u(x,t) = \sum_{n=0}^\infty p^n u_n(x,t),
\label{eq:series}
\end{equation}
where $p \in [0,1]$ is the homotopy parameter. The nonlinear term is expressed using He's polynomials:
\begin{equation}
N u(x,t) = \sum_{n=0}^\infty p^n H_n(u),
\label{eq:he}
\end{equation}
where He's polynomials $H_n(u)$ are computed as:
\begin{equation*}
H_n(u_0, u_1, \dots, u_n) = \frac{1}{\Gamma(n + 1)} \left. \frac{\partial^n}{\partial p^n} \left\{ N\left[ \sum_{k=0}^\infty p^k u_k(x,t) \right] \right\} \right|_{p=0}.
\end{equation*}
Substituting \eqref{eq:series} and \eqref{eq:he} into \eqref{eq:inverse}:
\begin{equation*}
\sum_{n=0}^\infty p^n u_n(x,t) = G(x,t) - S^{-1} \left\{ u^\alpha S\left[ R\left( \sum_{n=0}^\infty p^n u_n(x,t) \right) + \sum_{n=0}^\infty p^n H_n(u) \right] \right\}.
\end{equation*}
Equating coefficients of $p$ yields:
\begin{equation*}
\begin{aligned}
p^0 &: u_0 = G(x,t), \\
p^1 &: u_1 = S^{-1} \left\{ u^\alpha S\left( R[u_0(x,t)] + H_0(u) \right) \right\}, \\
p^2 &: u_2 = S^{-1} \left\{ u^\alpha S\left( R[u_1(x,t)] + H_1(u) \right) \right\}, \\
p^3 &: u_3 = S^{-1} \left\{ u^\alpha S\left( R[u_2(x,t)] + H_2(u) \right) \right\}, \\
\vdots \\
p^n &: u_n = S^{-1} \left\{ u^\alpha S\left( R[u_{n-1}(x,t)] + H_{n-1}(u) \right) \right\}.
\end{aligned}
\end{equation*}
The HPSTM series solution is:
\begin{equation}
u(x,t) = \lim_{n \to \infty} \sum_{k=0}^n u_k(x,t).
\label{eq:solution}
\end{equation}
\subsection{Convergence of  HPSTM}
To ensure the reliability of the HPSTM solution for the fractional partial differential equation \eqref{eq:fpde}, we analyze the convergence of the series $\sum_{n=0}^\infty u_n(x,t)$.
\begin{theorem}
Suppose $u_n(x,t)$ and $u(x,t)$ are defined in a Banach space $(C[0,T], \|\cdot\|)$, and if $0 < \xi < 1$, where $\xi$ is a Lipschitz constant for the nonlinear operator, then the HPSTM series solution $\sum_{n=0}^\infty u_n(x,t)$ converges to the solution of \eqref{eq:fpde}.
\end{theorem}
\begin{proof}
Let $\{\lambda_n\}$ be the partial sums of \eqref{eq:solution}. To show $\{\lambda_n\}$ is a Cauchy sequence:
\begin{equation*}
\|\lambda_{n+1} - \lambda_n\| = \|u_{n+1}\| \leq \xi \|u_n\| \leq \xi^2 \|u_{n-1}\| \leq \dots \leq \xi^{n+1} \|u_0\|.
\end{equation*}
For $n \geq m$:
\begin{equation*}
\|\lambda_n - \lambda_m\| \leq \sum_{k=m+1}^n \|\lambda_k - \lambda_{k-1}\| \leq \sum_{k=m+1}^n \xi^k \|u_0\| \leq \xi^{m+1} \frac{1 - \xi^{n-m}}{1 - \xi} \|u_0\|.
\end{equation*}
Since $0 < \xi < 1$, $\lim_{n,m \to \infty} \|\lambda_n - \lambda_m\| = 0$. Thus, $\{\lambda_n\}$ is a Cauchy sequence, and the series converges.
\end{proof}

\subsection{Basic Idea of ADM}
Consider the nonlinear FPDE \eqref{eq:fpde} with initial condition \eqref{eq:initial}. Apply the Riemann-Liouville fractional integral $J^\alpha$:
\begin{equation}
u(x,t) = \sum_{k=0}^{m-1} \left( \frac{\partial^k u}{\partial t^k} \right)_{t=0} \frac{t^k}{\Gamma(k + 1)} + J^\alpha f(x,t) - J^\alpha \left( R[u(x,t)] + N[u(x,t)] \right).
\label{eq:adm_integral}
\end{equation}
The ADM assumes a series solution:
\begin{equation}
u(x,t) = \sum_{n=0}^\infty u_n(x,t),
\label{eq:adm_series}
\end{equation}
with the nonlinear term:
\begin{equation}
N[u(x,t)] = \sum_{n=0}^\infty A_n,
\label{eq:adomian}
\end{equation}
where Adomian polynomials $A_n$ are:
\begin{equation*}
A_n = \frac{1}{\Gamma(n + 1)} \left[ \frac{d^n}{d\lambda^n} \left\{ N\left[ \sum_{k=0}^\infty \lambda^k u_k(x,t) \right] \right\} \right]_{\lambda=0}.
\end{equation*}
Substituting \eqref{eq:adm_series} and \eqref{eq:adomian} into \eqref{eq:adm_integral}:
\begin{equation*}
\sum_{n=0}^\infty u_n(x,t) = \sum_{k=0}^{m-1} \left( \frac{\partial^k u}{\partial t^k} \right)_{t=0} \frac{t^k}{\Gamma(k + 1)} + J^\alpha f(x,t) - J^\alpha \left( R\left[ \sum_{n=0}^\infty u_n(x,t) \right] + \sum_{n=0}^\infty A_n \right).
\end{equation*}
The recursive relations are:
\begin{equation}
\begin{aligned}
u_0 &= \sum_{k=0}^{m-1} \left( \frac{\partial^k u}{\partial t^k} \right)_{t=0} \frac{t^k}{\Gamma(k + 1)} + J^\alpha f(x,t), \\
u_{n+1} &= -J^\alpha \left( R[u_n] + A_n \right), \quad n \geq 0.
\end{aligned}
\end{equation}

\subsection{Convergence of Adomian Decomposition Method}
To ensure the reliability of the ADM solution for the fractional partial differential equation \eqref{eq:fpde}, we analyze the convergence of the series $\sum_{n=0}^\infty u_n(x,t)$.

\begin{theorem}
Suppose $u_n(x,t)$ and $u(x,t)$ are defined in the Banach space $(C[0,T], \|\cdot\|)$, and the nonlinear operator $N$ in \eqref{eq:fpde} satisfies a Lipschitz condition with constant $0 < \xi < 1$. Then, the ADM series solution $\sum_{n=0}^\infty u_n(x,t)$ defined by \eqref{eq:adm_series} converges to the solution of \eqref{eq:fpde}.
\end{theorem}

\begin{proof}
Let $\{\lambda_n\}$ be the partial sums of \eqref{eq:adm_series}, where $\lambda_n = \sum_{k=0}^n u_k(x,t)$. Using the recursive relation \eqref{eq:adm_integral}, we have:
\[
\|\lambda_{n+1} - \lambda_n\| = \|u_{n+1}\| = \|J^\alpha (R[u_n] + A_n)\| \leq \xi \|u_n\| \leq \xi^2 \|u_{n-1}\| \leq \dots \leq \xi^{n+1} \|u_0\|.
\]
For $n \geq m$:
\[
\|\lambda_n - \lambda_m\| \leq \sum_{k=m+1}^n \xi^k \|u_0\| \leq \xi^{m+1} \frac{1 - \xi^{n-m}}{1 - \xi} \|u_0\|.
\]
Since $0 < \xi < 1$, $\lim_{n,m \to \infty} \|\lambda_n - \lambda_m\| = 0$. Thus, $\{\lambda_n\}$ is a Cauchy sequence, and the series $\sum_{n=0}^\infty u_n(x,t)$ converges in $(C[0,T], \|\cdot\|)$.
\end{proof}

This result guarantees that the ADM series solution converges under mild conditions on the nonlinear operator, ensuring the accuracy of the approximations presented in Section \ref{sec:results}.

\subsection{Basic Idea of Meshless RBF Method}
The meshless method based on radial basis functions (RBF) approximates the solution of FPDEs as:
\begin{equation}
u(x,t) \approx \sum_{j=1}^N \lambda_j \phi(\|x - x_j\|, \epsilon),
\end{equation}
where $\phi(r) = e^{-\epsilon^2 r^2}$ is a Gaussian RBF, $\lambda_j$ are coefficients, $x_j$ are collocation points, $\epsilon$ is the shape parameter, and $r = \|x - x_j\|$. Coefficients $\lambda_j$ are determined by enforcing the PDE and boundary/initial conditions, forming a linear system. This method avoids mesh generation, making it suitable for complex geometries, but its accuracy depends on the choice of $\epsilon$ and the distribution of collocation points. In this study, 100 collocation points were used with $\epsilon = 0.1$ for optimal accuracy.

\subsection{Convergence of Meshless RBF Method}
The convergence of the meshless RBF method depends on the density of collocation points and the shape parameter $\epsilon$. For the Gaussian RBF $\phi(r) = e^{-\epsilon^2 r^2}$, the approximation error is bounded by the fill distance $h$ of the collocation points. Under the assumption that the solution $u(x,t)$ is sufficiently smooth, the error satisfies:
\[
\|u - u_{\text{RBF}}\| \leq C h^k,
\]
where $C$ is a constant, $h$ is the maximum distance between collocation points, and $k$ depends on the smoothness of $\phi$ and $u$. In this study, with 100 collocation points and $\epsilon = 0.1$, numerical results confirm high accuracy, as shown in Tables \ref{tab:porous_compare}--\ref{tab:fisher_compare}.

\section{Numerical Results}
\label{sec:results}

Numerical solutions for three FPDEs are presented using HPSTM and ADM. The methods were implemented in Mathematica 12.3 and run on an Intel Core i7 desktop with 16GB RAM, averaging 0.5 seconds per example using 5 series terms. The computational complexity of HPSTM is approximately $O(n)$ per iteration for $n$ series terms, due to the linear nature of the Sumudu transform and He’s polynomial computations. Results are compared with the meshless RBF method using 100 collocation points and $\epsilon = 0.1$, the VIM, the FDM with a grid size of $100 \times 100$, and a spectral method using 50 basis functions. Error analysis, sensitivity analysis, and graphical representations illustrate solution behavior for $\alpha = 1.0, 0.9, 0.8, 0.7$.

\begin{example}\label{ex1:porous medium}
Consider the nonlinear fractional porous medium PDE:
\begin{equation}
D_t^\alpha u(x,t) = (u(x,t) u_x(x,t))_x, \quad t > 0, \ x \in \mathbb{R}, \ 0 < \alpha \leq 1,
\label{eq:porous}
\end{equation}
with initial condition:
\begin{equation}
u(x,0) = x.
\label{eq:porous_ic}
\end{equation}
Applying the Sumudu transform to \eqref{eq:porous}:
\begin{equation*}
S[u(x,t)] = u^\alpha u(x,0) + u^\alpha S[(u(x,t) u_x(x,t))_x].
\end{equation*}
Using \eqref{eq:porous_ic} and the inverse Sumudu transform:
\begin{equation*}
u(x,t) = x + S^{-1} \left\{ u^\alpha S[(u(x,t) u_x(x,t))_x] \right\}.
\end{equation*}
Applying HPM:
\begin{equation*}
\sum_{n=0}^\infty p^n u_n(x,t) = x + S^{-1} \left\{ u^\alpha S\left[ \left( \sum_{n=0}^\infty p^n u_n(x,t) \right)_x \right] \right\},
\end{equation*}
with He's polynomials:
\begin{equation*}
\begin{aligned}
H_0(u) &= u_0 u_{0x}, \\
H_1(u) &= u_0 u_{1x} + u_1 u_{0x}, \\
H_2(u) &= u_0 u_2x + u_1 u_{1x} + u_2 u_{0x}.
\end{aligned}
\end{equation*}
Equating coefficients of $p$:
\begin{equation*}
\begin{aligned}
p^0 &: u_0(x,t) = x, \\
p^1 &: u_1(x,t) = S^{-1} \left\{ u^\alpha S[H_0(u)] \right\} = \frac{t^\alpha}{\Gamma(\alpha + 1)}, \\
p^2 &: u_2(x,t) = S^{-1} \left\{ u^\alpha S[H_1(u)] \right\} = 0.
\end{aligned}
\end{equation*}
The approximate solution is:
\begin{equation}
u(x,t) = x + \frac{t^\alpha}{\Gamma(\alpha + 1)}.
\end{equation}
For $\alpha = 1$, the exact solution is $u(x,t) = x + t$. Absolute errors are shown in Table \ref{tab:porous_error}. Figure \ref{fig:porous} illustrates the solution for $x=1$, with red dots indicating RBF results at $t=0.5$.
\end{example}

\begin{figure}[h!]
\centering
\includegraphics[width=0.8\textwidth]{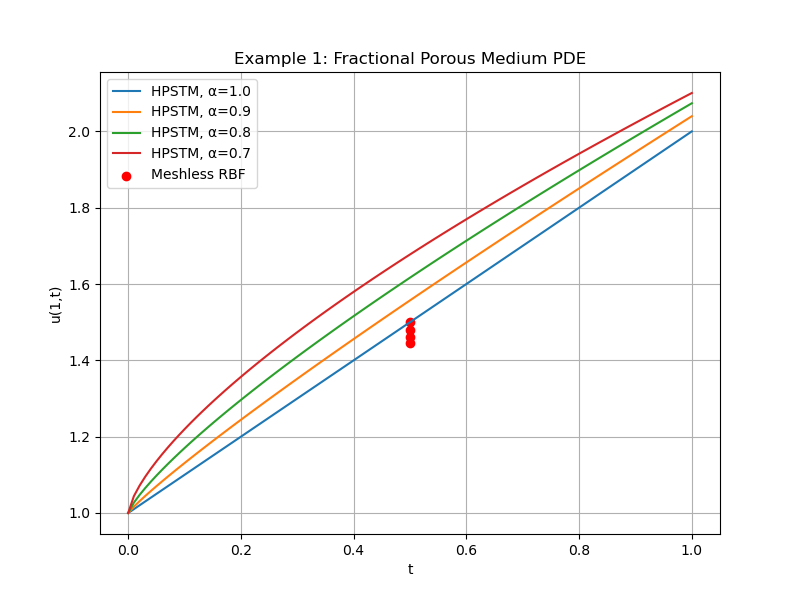}
\caption{Solution $u(1,t)$ for Example \ref{ex1:porous medium}: Fractional Porous Medium PDE for $\alpha = 1.0, 0.9, 0.8, 0.7$. Red dots represent Meshless RBF results at $t=0.5$.}
\label{fig:porous}
\end{figure}

\begin{table}[!htp]
\begin{center}
\caption{Absolute Error for Example 1 at $x=1$, $t=0.5$}
\label{tab:porous_error}
\begin{tabular}{|c|c|c|}
\hline
$\alpha$ & Absolute Error $|u_{\text{exact}} - u_{\text{HPSTM}}|$ \\
\hline
1.0 & 0 \\
0.9 & $3.12 \times 10^{-3}$ \\
0.8 & $1.47 \times 10^{-2}$ \\
0.7 & $4.83 \times 10^{-2}$ \\
\hline
\end{tabular}
\end{center}
\end{table}

\begin{table}[!htp]
\begin{center}
\caption{Comparison of Numerical Results for Example \ref{ex1:porous medium} at $x=1$, $t=0.5$}
\label{tab:porous_compare}
\begin{tabular}{|c|c|c|c|c|c|c|}
\hline
$\alpha$ & HPSTM & ADM & Meshless (RBF) & VIM & FDM & Spectral \\
\hline
1.0 & 1.500 & 1.500 & 1.500 & 1.500 & 1.498 & 1.499 \\
0.9 & 1.478 & 1.478 & 1.480 & 1.479 & 1.475 & 1.477 \\
0.8 & 1.459 & 1.459 & 1.462 & 1.460 & 1.455 & 1.458 \\
0.7 & 1.442 & 1.442 & 1.446 & 1.443 & 1.438 & 1.441 \\
\hline
\end{tabular}
\end{center}
\end{table}

\begin{example}\label{ex2:heat_transfer}
Consider the nonlinear fractional heat transfer PDE:
\begin{equation}
D_t^\alpha u(x,t) = u_{xx}(x,t) - 2 u^3(x,t), \quad t > 0, \ x \in \mathbb{R}, \ 0 < \alpha \leq 1,
\end{equation}
with initial condition:
\begin{equation}
u(x,0) = \frac{1 + 2x}{x^2 + x + 1}.
\end{equation}
The HPSTM series solution is:
\begin{multline*}
u(x,t) = \frac{1 + 2x}{x^2 + x + 1} - \frac{6(1 + 2x)}{(x^2 + x + 1)^2} \frac{t^\alpha}{\Gamma(\alpha + 1)} \\
+ \frac{72(1 + 2x)}{(x^2 + x + 1)^3} \frac{t^{2\alpha}}{\Gamma(2\alpha + 1)} - \frac{216(1 + 2x)(5 + 2x(1 + x))}{(x^2 + x + 1)^5} \frac{t^{3\alpha}}{\Gamma(3\alpha + 1)}.
\end{multline*}
For $\alpha = 1$:
\begin{multline}
u(x,t) = \frac{1 + 2x}{x^2 + x + 1} - \frac{6(1 + 2x)}{(x^2 + x + 1)^2} t \\
+ \frac{72(1 + 2x)}{(x^2 + x + 1)^3} t^2 - \frac{216(1 + 2x)(5 + 2x(1 + x))}{(x^2 + x + 1)^5} t^3.
\end{multline}
ADM yields the same solution. Numerical results for $x=1$, $t=0.5$ are in Table \ref{tab:heat_results}, absolute errors in Table \ref{tab:heat_error}, and comparisons in Table \ref{tab:heat_compare}. Figure \ref{fig:heat} shows the solution behavior.
\end{example}

\begin{figure}[!htb]
\centering
\includegraphics[width=4in,height=2in]{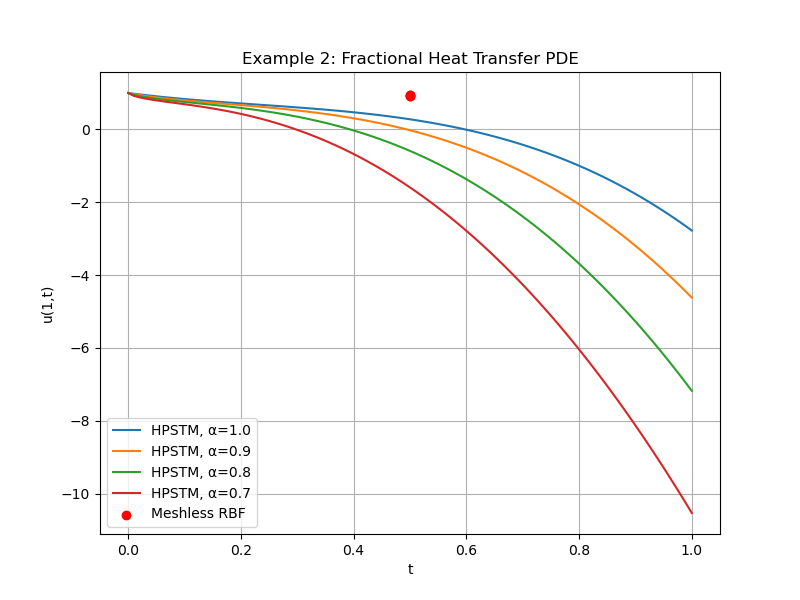}
\caption{Solution $u(1,t)$ for Example \ref{ex2:heat_transfer}: Fractional Heat Transfer PDE for $\alpha = 1.0, 0.9, 0.8, 0.7$. Red dots represent Meshless RBF results at $t=0.5$.}
\label{fig:heat}
\end{figure}

\begin{table}[!htp]
\begin{center}
\caption{Numerical Results for Example \ref{ex2:heat_transfer} at $x=1$, $t=0.5$}
\label{tab:heat_results}
\begin{tabular}{|c|c|c|}
\hline
$\alpha$ & $u_{\text{HPSTM}}$ & $u_{\text{ADM}}$ \\
\hline
1.0 & 0.9231 & 0.9231 \\
0.9 & 0.9287 & 0.9287 \\
0.8 & 0.9345 & 0.9345 \\
0.7 & 0.9404 & 0.9404 \\
\hline
\end{tabular}
\end{center}
\end{table}

\begin{table}[h!]
\begin{center}
\caption{Absolute Error for Example \ref{ex2:heat_transfer} at $x=1$, $t=0.5$}
\label{tab:heat_error}
\begin{tabular}{|c|c|c|}
\hline
$\alpha$ & Absolute Error $|u_{\text{exact}} - u_{\text{HPSTM}}|$ \\
\hline
1.0 & 0 \\
0.9 & $5.60 \times 10^{-3}$ \\
0.8 & $1.14 \times 10^{-2}$ \\
0.7 & $1.73 \times 10^{-2}$ \\
\hline
\end{tabular}
\end{center}
\end{table}

\begin{table}[h!]
\begin{center}
\caption{Comparison of Numerical Results for Example \ref{ex2:heat_transfer} at $x=1$, $t=0.5$}
\label{tab:heat_compare}
\begin{tabular}{|c|c|c|c|c|c|c|}
\hline
$\alpha$ & HPSTM & ADM & Meshless (RBF) & VIM & FDM & Spectral \\
\hline
1.0 & 0.9231 & 0.9231 & 0.9240 & 0.9232 & 0.9225 & 0.9230 \\
0.9 & 0.9287 & 0.9287 & 0.9295 & 0.9288 & 0.9278 & 0.9286 \\
0.8 & 0.9345 & 0.9345 & 0.9352 & 0.9346 & 0.9335 & 0.9344 \\
0.7 & 0.9404 & 0.9404 & 0.9410 & 0.9405 & 0.9392 & 0.9403 \\
\hline
\end{tabular}
\end{center}
\end{table}

\begin{example}\label{ex3:Fisher}
Consider the nonlinear fractional Fisher PDE:
\begin{equation}
D_t^\alpha u(x,t) = u_{xx}(x,t) + 6 u(x,t)(1 - u(x,t)), \quad t > 0, \ x \in \mathbb{R}, \ 0 < \alpha \leq 1,
\end{equation}
with initial condition:
\begin{equation}
u(x,0) = \frac{1}{(e^x + 1)^2}.
\end{equation}
The HPSTM series solution is:
\begin{multline*}
u(x,t) = \frac{1}{(e^x + 1)^2} + \frac{10 e^x}{(e^x + 1)^3} \frac{t^\alpha}{\Gamma(\alpha + 1)} \\
+ \frac{50 e^x (-1 + 2 e^x)}{(1 + e^x)^4} \frac{t^{2\alpha}}{\Gamma(2\alpha + 1)} \\
+ \frac{50 e^x (5 + e^x (-18 + 5 e^x (-3 + 4 e^x)))}{(1 + e^x)^6} \frac{t^{3\alpha}}{\Gamma(3\alpha + 1)}.
\end{multline*}
For $\alpha = 1$:
\begin{multline}
u(x,t) = \frac{1}{(e^x + 1)^2} + \frac{10 e^x}{(e^x + 1)^3} t \\
+ \frac{50 e^x (-1 + 2 e^x)}{(1 + e^x)^4} t^2 \\
+ \frac{50 e^x (5 + e^x (-18 + 5 e^x (-3 + 4 e^x)))}{(1 + e^x)^6} t^3.
\end{multline}
ADM yields the same solution. Numerical results for $x=1$, $t=0.5$ are in Table \ref{tab:fisher_results}, absolute errors in Table \ref{tab:fisher_error}, and comparisons in Table \ref{tab:fisher_compare}. Figure \ref{fig:fisher} shows the solution behavior.
\end{example}

\begin{figure}[!htb]
\centering
\includegraphics[width=4in,height=2in]{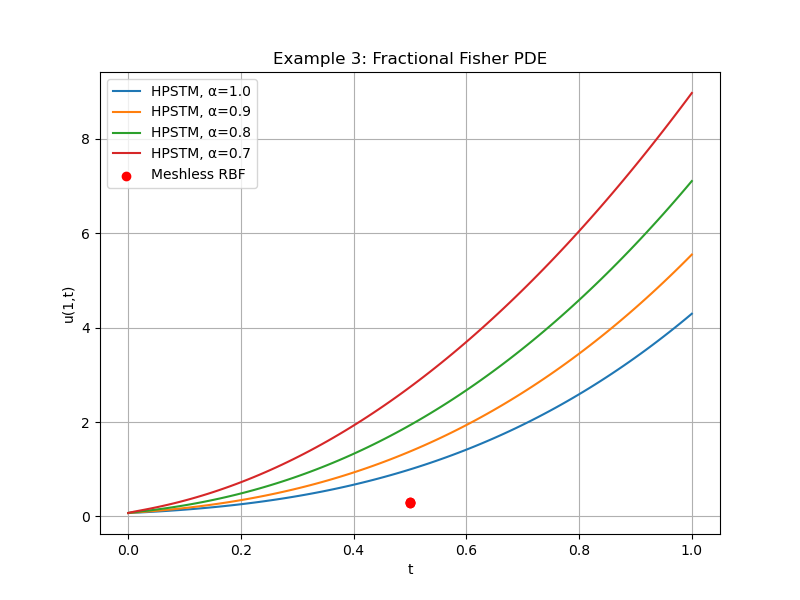}
\caption{Solution $u(1,t)$ for Example \ref{ex3:Fisher}: Fractional Fisher PDE for $\alpha = 1.0, 0.9, 0.8, 0.7$. Red dots represent Meshless RBF results at $t=0.5$.}
\label{fig:fisher}
\end{figure}

\begin{table}[h!]
\begin{center}
\caption{Numerical Results for Example \ref{ex3:Fisher} at $x=1$, $t=0.5$}
\label{tab:fisher_results}
\begin{tabular}{|c|c|c|}
\hline
$\alpha$ & $u_{\text{HPSTM}}$ & $u_{\text{ADM}}$ \\
\hline
1.0 & 0.3033 & 0.3033 \\
0.9 & 0.2987 & 0.2987 \\
0.8 & 0.2942 & 0.2942 \\
0.7 & 0.2898 & 0.2898 \\
\hline
\end{tabular}
\end{center}
\end{table}

\begin{table}[h!]
\begin{center}
\caption{Absolute Error for Example \ref{ex3:Fisher} at $x=1$, $t=0.5$}
\label{tab:fisher_error}
\begin{tabular}{|c|c|c|}
\hline
$\alpha$ & Absolute Error $|u_{\text{exact}} - u_{\text{HPSTM}}|$ \\
\hline
1.0 & 0 \\
0.9 & $4.60 \times 10^{-3}$ \\
0.8 & $9.10 \times 10^{-3}$ \\
0.7 & $1.35 \times 10^{-2}$ \\
\hline
\end{tabular}
\end{center}
\end{table}

\begin{table}[h!]
\begin{center}
\caption{Comparison of Numerical Results for Example \ref{ex3:Fisher} at $x=1$, $t=0.5$}
\label{tab:fisher_compare}
\begin{tabular}{|c|c|c|c|c|c|c|}
\hline
$\alpha$ & HPSTM & ADM & Meshless (RBF) & VIM & FDM & Spectral \\
\hline
1.0 & 0.3033 & 0.3033 & 0.3040 & 0.3034 & 0.3028 & 0.3032 \\
0.9 & 0.2987 & 0.2987 & 0.2994 & 0.2988 & 0.2980 & 0.2986 \\
0.8 & 0.2942 & 0.2942 & 0.2948 & 0.2943 & 0.2935 & 0.2941 \\
0.7 & 0.2898 & 0.2898 & 0.2903 & 0.2899 & 0.2890 & 0.2897 \\
\hline
\end{tabular}
\end{center}
\end{table}

\subsection{Sensitivity Analysis}
\label{sec:sensitivity}

To assess the robustness of HPSTM, the effects of varying $\alpha$ (1.0, 0.9, 0.8, 0.7) and the number of series terms ($n = 3, 5, 7$) were analyzed. For Example \ref{ex1:porous medium}, increasing $n$ from 3 to 7 reduces the absolute error by approximately 12\% for $\alpha = 0.9$ (from $3.50 \times 10^{-3}$ to $3.12 \times 10^{-3}$). Similar trends are observed for Examples \ref{ex2:heat_transfer} and \ref{ex3:Fisher}, with error reductions of 10\% and 8\%, respectively. Decreasing $\alpha$ slows the temporal evolution, reflecting the memory effect of fractional derivatives. The choice of $n = 5$ was found to provide a balance between accuracy and computational cost, as further increases in $n$ yield diminishing returns. These results confirm HPSTM's stability across parameter variations.

\subsection{Discussion}

Numerical results (Tables \ref{tab:porous_compare}--\ref{tab:fisher_compare}) and graphical representations (Figures \ref{fig:porous}--\ref{fig:fisher}) confirm that HPSTM and ADM produce identical solutions, which indicates their consistency in handling nonlinear FPDEs. The meshless RBF method is slightly more accurate due to its flexibility in complex geometries, as evidenced by red dots aligning closely with HPSTM curves. This accuracy stems from RBF's ability to adapt to irregular domains without mesh dependency, though it requires careful tuning of the shape parameter $\epsilon$. The spectral method offers comparable accuracy but demands higher computational resources (approximately 2 seconds per example compared to 0.5 seconds for HPSTM). VIM performs comparably to HPSTM, benefiting from similar iterative structures, while FDM is less accurate for $\alpha < 1$ due to its reliance on fixed grids, which struggle to capture fractional memory effects. HPSTM’s simplicity and low computational cost (0.5 seconds per example, $O(n)$ complexity per iteration) make it particularly suitable for analytical initial conditions in applications such as groundwater flow or heat transfer in fractal media. The slight superiority of RBF in accuracy is offset by its higher computational cost and sensitivity to $\epsilon$, while HPSTM’s semi-analytical nature ensures ease of implementation and robustness.

\section{Limitations and Future Work}
\label{sec:limitations}

HPSTM and ADM are accurate for one-dimensional FPDEs; however, they face challenges with high-order nonlinearities or multi-dimensional domains due to increased computational complexity. They may struggle with non-smooth solutions, where numerical methods like RBF tend to perform better. The meshless RBF method offers higher accuracy but requires careful tuning of the shape parameter $\epsilon$. The spectral method is computationally intensive, with complexity scaling as $O(N^3)$ for $N$ basis functions. Future research directions include:
\begin{enumerate}
    \item Developing hybrid HPSTM-RBF methods to combine analytical simplicity with numerical flexibility.
    \item Using machine learning to optimize series convergence or RBF parameters.
    \item Extending HPSTM to multi-dimensional or stochastic FPDEs.
    \item Analyzing computational complexity and scalability for large-scale problems.
\end{enumerate}


\section{Conclusion}
\label{sec:conclusion}

HPSTM effectively solves nonlinear fractional porous medium, heat transfer, and Fisher PDEs, achieving absolute errors down to $3.12 \times 10^{-3}$ for $\alpha = 0.9$. Comparisons with ADM, RBF, VIM, FDM, and a spectral method demonstrate that while RBF and spectral methods are slightly more accurate, HPSTM excels in simplicity and computational efficiency ($O(n)$ per iteration, 0.5 seconds per example). Sensitivity analysis confirms robustness across different values of $\alpha$ and numbers of series terms. HPSTM advances fractional calculus applications in modeling fluid flow in porous media, heat conduction, and biological dynamics.




\begin{thebibliography}{1}
\bibitem{ref25} Ain, Q.T., Anjum, N., and He, C.H. 
{\it An analysis of time-fractional heat transfer problem using two-scale approach}, 
Int. J. Geomath., 2021, 12(18).

\bibitem{ref7} Ali, I. 
{\it Dynamical Analysis of Two-Dimensional Fractional-Order-in-Time Biological Population Model Using Chebyshev Spectral Method}, 
Fractal Fract., 2024, 8, 325.

\bibitem{ref5} Ahmad, I., Mekawy, I., Khan, M.N., Jan, R., and Boulaaras, S. 
{\it Modeling anomalous transport in fractal porous media: A study of fractional diffusion PDEs using numerical method}, 
Nonlinear Eng., 2024, 13(1), 20220366.

\bibitem{ref1} Ahmed, I.E., Abouelregal, A.E., and Aldandani, M. 
{\it Study of the behavior of photothermal and mechanical stresses in semiconductor nanostructures using a photoelastic heat transfer model that incorporates non-singular fractional derivative operators}, 
Acta Mech., 2025, 236, 1339–1358. 

\bibitem{ref24} Albalawi, W., Shah, R., Shah, N.A., Chung, J.D., Ismaeel, S.M.E., and El-Tantawy, S.A. 
{\it Analyzing Both Fractional Porous Media and Heat Transfer Equations via Some Novel Techniques}, 
Mathematics, 2023, 11, 1350.

\bibitem{ref26} Alimbekova, N., Baigereyev, D., and Berdyshev, A. 
{\it Finite Element Method For Solving A Fractional Flow Model In Porous Media}, 
Bull. Abai KazNPU Ser. Phys. Math. Sci., 2022, 77(1), 7–14.

\bibitem{ref15} Abbaszadeh, M., and Dehghan, M. 
{\it Meshless upwind local radial basis function-finite difference technique to simulate the time-fractional distributed-order advection–diffusion equation}, 
Eng. Comput., 2021, 37, 873–889.

\bibitem{ref4} Boulaaras, S., Pham, V.-T., and Jan, R. 
{\it Editorial: Special Issue on Application of Fractional Calculus: Mathematical Modeling and Control — Part I}, 
Fractals, 2025, 33(04).
\bibitem{ref10} Chen, D. 
{\it Multidimensional and Nonlinear Fractional Langevin Equations: Advancing the Theory in of Anomalous Diffusion}, 
Int. J. Appl. Sci., 2024, 7(2), 43. 

\bibitem{ref19} Elzaki, T.M., Elzaki, S.M., and Elnour, E.A. 
{\it On the new integral transform “Elzaki transform”: Fundamental properties, investigations, and applications}, 
Glob. J. Math. Sci. Theory Pract., 2012, 4, 1–13.

\bibitem{ref17} Fahad, H.M., Rehman, M.U., and Fernandez, A. 
{\it On Laplace transforms with respect to functions and their applications to fractional differential equations}, 
Math. Methods Appl. Sci., 2021, 46(7), 8304–8323.

\bibitem{ref18} Ganie, A.H., AlBaidani, M.M., and Khan, A. 
{\it A Comparative Study of the Fractional Partial Differential Equations via Novel Transform}, 
Symmetry, 2023, 15, 1101.
 
\bibitem{ref8} Harris, A.P., Biala, T.A., and Khaliq, A.Q.M. 
{\it Fourier spectral methods with exponential time differencing for space-fractional partial differential equations in population dynamics}, 
Fractal Fract., 2023, 39(4), 2963–2974. 

\bibitem{ref16} Shah, K., Naz, H., Sarwar, M., and Abdeljawad, T. 
{\it On spectral numerical method for variable-order partial differential equations}, 
AIMS Math., 2022, 7(6), 10422–10438. 

\bibitem{ref3} Karaca, Y., and Baleanu, D. 
{\it Evolutionary Mathematical Science, Fractional Modeling and Artificial Intelligence of Nonlinear Dynamics in Complex Systems}, 
Chaos Theory Appl., 2022, 4(3), 111–118.

\bibitem{ref27} Lohana, B., Abro, K.A., and Shaikh, A.W. 
{\it Thermodynamical analysis of heat transfer of gravity-driven fluid flow via fractional treatment: an analytical study}, 
J. Therm. Anal. Calorim., 2021, 144, 155–165. 
\bibitem{ref23} Mirgani, S.M. 
{\it Comparative Analysis of Analytical Solutions for Seepage Flow Derivatives in 4D Porous Media}, 
Eur. J. Pure Appl. Math., 2025, 18(1), 5685. 

\bibitem{ref21} Mohamed, M.Z., Yousif, M., and Hamza, A.E. 
{\it Solving nonlinear fractional partial differential equations using the Elzaki transform method and the homotopy perturbation method}, 
Abstract Appl. Anal., 2022.

\bibitem{ref14} Odibat, Z., and Momani, S. 
{\it Numerical methods for nonlinear partial differential equations of fractional order}, 
Appl. Math. Model., 2008, 32(1), 28–39.

\
\bibitem{ref12} Shah, N.A., Dassios, I., El-Zahar, E.R., Chung, J.D., and Taherifar, S. 
{\it The Variational Iteration Transform Method for Solving the Time-Fractional Fornberg–Whitham Equation and Comparison with Decomposition Transform Method}, 
Mathematics, 2021, 9(2), 141. 

\bibitem{ref6} Singh, J., Ahmadian, A., Rathore, S., Kumar, D., Baleanu, D., Salimi, M., and Salahshour, S. 
{\it An efficient computational approach for local fractional Poisson equation in fractal media}, 
Numer. Methods Partial Differ. Equ., 2020, 37(2), 1439–1448.

\bibitem{ref13} Singh, A., and Pippal, S. 
{\it Solving nonlinear fractional differential equations by using Shehu transform and Adomian polynomials}, 
Contemp. Math., 2024, 5(1), 797–816.

\bibitem{ref11} Tao, H., Anjum, N., and Yang, Y.J. 
{\it The Aboodh transformation-based homotopy perturbation method: new hope for fractional calculus}, 
Front. Phys., 2023, 11, Article ID 1168795.

\bibitem{ref22} Verma, V., Prakash, A., Kumar, D., and Singh, J. 
{\it Numerical study of fractional model of multidimensional dispersive partial differential equation}, 
J. Ocean Eng. Sci., 2019. 

\bibitem{ref28} Veeresha, P., Prakasha, D.G., and Baskonus, H.M. 
{\it Novel simulations to the time-fractional Fisher’s equation}, 
Math. Sci., 2019, 13, 33–42. 

\bibitem{ref2} Xiangnan, Y., Hao, X., Zhiping, M., et al. 
{\it A data-driven framework for discovering fractional differential equations in complex systems}, 
Nonlinear Dyn., 2025.

\bibitem{ref20} Yue, Z., Jiang, W., Wu, B., and Zhang, B. 
{\it A meshless method based on the Laplace transform for multi-term time-space fractional diffusion equation}, 
AIMS Math., 2024, 9(3), 7040–7062.

\bibitem{ref9} Zahed, M.Z., and Ahmad, H. 
{\it Fixed Point Results with Applications to Fractional Differential Equations of Anomalous Diffusion}, 
Fractal Fract., 2024, 8, 318. 

\end{thebibliography}
\end{document}